\documentclass{article}

\usepackage[letterpaper,top=2cm,bottom=2cm,left=2.9cm,right=2.9cm,marginparwidth=1.75cm]{geometry}

\usepackage[colorlinks=true, allcolors=blue]{hyperref}
\usepackage{amssymb}
\usepackage[utf8]{inputenc}
\usepackage{amsthm}
\usepackage{amsmath}
\usepackage{amsfonts}
\usepackage{enumerate}
\usepackage[numeric]{amsrefs}
\usepackage{array}
\usepackage{tabularx}
\usepackage{graphicx}
\usepackage{xcolor}
\usepackage{soul}
\usepackage{booktabs}
\usepackage[scaled=0.95]{helvet}
\usepackage{float}

\definecolor{softyellow}{RGB}{255,160,100}
\definecolor{darkergreen}{rgb}{0.0, 0.6, 0.0}

\newcolumntype{L}[1]{>{\raggedright\let\newline\\\arraybackslash\hspace{0pt}}m{#1}}
\newcolumntype{C}[1]{>{\centering\let\newline\\\arraybackslash\hspace{0pt}}m{#1}}
\newcolumntype{R}[1]{>{\raggedleft\let\newline\\\arraybackslash\hspace{0pt}}m{#1}}

\usepackage{mathrsfs}
\usepackage{url}
\usepackage[all]{xy}
\usepackage[scr]{rsfso}

\newtheorem{thm}{Theorem}

\newtheorem{lemma}[thm]{Lemma}

\newtheorem*{thm*}{Theorem}
\newtheorem*{conj*}{Conjecture}
\newtheorem*{qn*}{Question}

\theoremstyle{definition}

\newtheorem*{recentprogress}{Recent Progress}
\newtheorem{defi}{Definition}
\newtheorem{claim}{Claim}

\newcolumntype{D}{ >{\centering\arraybackslash} m{0.7cm} }
\newcolumntype{E}{ >{\centering\arraybackslash} m{2.5cm} }
\newcolumntype{F}{ >{\centering\arraybackslash} m{0.7cm} }
\newcolumntype{G}{ >{\centering\arraybackslash} m{5cm} }

\renewenvironment{abstract}
  {\small
   \begin{center}
   \begin{minipage}{0.73\textwidth}
   \noindent\textsc{Abstract.}\ }
  {\end{minipage}
   \end{center}}

\newcommand{\HTwoBound}{31\,140}
\newcommand{\HThreeBound}{2\,560\,022}
\newcommand{\HFourBound}{193\,672\,056}
\newcommand{\HFiveBound}{13\,659\,045\,840}

\newcommand{\MTwoBound}{3\,438}
\newcommand{\MThreeBound}{197\,454}
\newcommand{\MFourBound}{10\,995\,550}
\newcommand{\MFiveBound}{607\,150\,000}

\title{\large\bfseries\MakeUppercase{Clusters of primes revisited}}
\author{\fontsize{10}{12}\selectfont\scshape Paweł Lewulis}
\date{}

\begin{document}
\maketitle

\begin{abstract}
    We revisit the generalized Rayleigh quotient optimization problem related to the Maynard-Tao sieve with the goal of finding a more efficient function basis. We demonstrate that a family based on products of reciprocals, previously used in asymptotic estimates for $H_m = \liminf_{n \rightarrow \infty} (p_{n+m} - p_n)$, provides results that are close to the theoretical limits of the sieve across many regimes. In particular, we obtain:
    \begin{equation*}
        H_2 \leq \HTwoBound, \quad 
        H_3 \leq \HThreeBound, \quad 
        % 2563692
        H_4 \leq \HFourBound, \quad
        H_5 \leq \HFiveBound.
    \end{equation*}
    We also show that, as $k \to \infty$, the related variational quantity $M_k$, introduced by Maynard \cite{Maynard}, satisfies the lower bound
    \begin{equation*}
        M_k \geq \log k - 0.273572 + o(1).
    \end{equation*}
\end{abstract}

\section{Introduction}

In this paper, we study upper bounds for $H_m := \liminf_{n \to \infty}(p_{n+m} - p_n)$. The first unconditional finite bound of this kind was established by Zhang \cite{Zhang}, who proved that $H_1 \leq 70\,000\,000$. Soon afterward, Maynard \cite{Maynard} developed a different approach and obtained $H_1 \leq 600$ and $H_m \ll m^3 \exp(4m)$. Both of these breakthroughs led to successive phases of the Polymath8 project, which aimed to push the newly developed techniques to their limits \cite{PM8a, PM8b}. The bounds $H_1 \leq 246$, $H_2 \leq 398\,130$, $H_3 \leq 24\,797\,814$ (and similar for larger $m$) were established and remained untouched for many years.

Eventually, Julia Stadlmann obtained slightly stronger bounds for $H_m$ for $m = 2, 3, 4, 5$ as a corollary of her result concerning primes in arithmetic progressions \cite{JuliaS_Zhang_paper}:
\begin{equation*}
    H_2 \leq 396 \, 504, \quad H_3 \leq 24 \, 407 \,016, 
    \quad H_4 \leq 1 \, 391 \, 051 \, 532, 
    \quad H_5 \leq 77 \, 510 \, 685 \, 234.
\end{equation*}
Recently, she also established $H_1 \leq 240$ by finding a way to combine ideas from  Maynard's and Zhang's approaches \cite{JuliaS}. A concurrent preprint by Hanxin Zhang \cite{HanxinZhang2026} reported the same bound. Several further preprints claim improvements to the bound for $H_1$ using various strategies \cite{Kintali2026, SongYue2026, Charton2026, OAI}.

\begin{recentprogress} 
    This paper was developed independently of the recent preprint by Roeling \cite{Roeling}, of which the author became aware only when this project was already at an advanced stage. The methods used are different, and the results obtained here are slightly stronger.

    Very recently, another preprint \cite{Kapila} reports further improvements to the numerical bounds on $H_m$ for $m=2,3,4,5$ by incorporating enlarged sieve support and the methods developed in \cite{OAI}. The resulting bounds are stronger than those presented here, in particular $H_2 \leq 20 \, 370$.

    Concerning the asymptotic results, Theorem \ref{asymptotyczny_teoremat} remains unmatched in the current literature, as far as the author is aware.
\end{recentprogress}

The starting point in all modern approaches to the problem of short gaps between primes is the GPY setup \cite{GPY} with some variant of the Selberg sieve, usually the Maynard–Tao sieve. We begin with an expression of the form
\begin{equation}\label{GPY_setup}
    \mathcal{S} (x; k, m, \mathcal{H}, \Lambda^2_{\mathrm{Sel}}) = \sum_{x<n \leq 2x} \left( \sum_{h \in \mathcal{H}} \mathbf{1}_{\mathbf{P}} 
    (n + h)   - m \right) \, \Lambda^2_{\mathrm{Sel}} (n; \mathcal{H}),
\end{equation}
where $\mathcal{H}$ is an admissible $k$-tuple and $\Lambda^2_{\mathrm{Sel}}$ are appropriately chosen non-negative weights that are designed to detect good candidates for primes. If we can demonstrate that there exist arbitrarily large
values of $x$ such that
\begin{equation*}
    \mathcal{S} > 0,
\end{equation*}
then an upper bound for $H_m$ follows. Of course, ideally $\Lambda^2_{\mathrm{Sel}}$ would detect primes exclusively, but even for $k=2$ this would require arithmetic information well beyond our reach (the parity barrier). In practice, the best we can hope for with currently known sieve techniques is to retain a sufficiently high density of primes among the remaining candidates. This can lead to non-trivial bounds on $H_m$.

% In \cite{PM8b} studying the setup of the form (\ref{GPY_setup}) led to the following definitions and 

\subsection*{From primes to functional optimization}

We now introduce the necessary terminology and briefly describe how the Maynard-Tao sieve transforms prime detection into an optimization problem. Starting from an expression of the form (\ref{GPY_setup}), we expand the Selberg weights, swap the sums, and obtain something of the form:
\begin{equation}\label{general_problem_form}
    \mathcal{S} = 
    \mathrm{MT_{\mathbf{1}_\mathbf{P}}} - m \cdot \mathrm{MT_{1}} 
    + (\text{BV-type error}),
\end{equation}
where the main terms give rise to certain integral functionals to be optimized later. The error term is closely related to the distributional claims concerning primes in arithmetic progressions of the Bombieri-Vinogradov type. Better results in this area lead to larger sieve supports. 

Recall some key notation and concepts. We closely follow the notation introduced by Polymath8b.

\begin{defi}[Standard simplex]
    For any $\theta > 0$ and any integer $k \geq 1$, we define the rescaled standard simplex as the set
    \begin{equation*}
        \theta \cdot \mathcal{R}_k :=
        \left\{(t_1, \ldots ,t_k) \in [0,+\infty)^k \colon
        t_1 + \cdots + t_k \le \theta \right\}.
    \end{equation*}
    If $\theta=1$, then we simply write $\mathcal{R}_k$.
\end{defi}

\begin{defi}[Key functionals]\label{key_functionals}
    Let $k \ge 2$ be an integer. Let
    $0 \leq \varepsilon < 1$ be a real parameter. Let also $F \colon [0,+\infty)^k \to \mathbf{R}$ be a locally square-integrable, symmetric function. We define the functionals:
    \begin{align*}
        I_{\varepsilon}(F) &:= 
        \int \limits_{(1 + \varepsilon) \cdot \mathcal{R}_{k}} F(t_1,\dots,t_k)^2 \,dt_1\cdots dt_k, \\
        J_{\varepsilon}(F) &:=
        \int \limits_{(1 - \varepsilon) \cdot \mathcal{R}_{k-1}}
        \left(
            \int \limits_0^{\infty} F(t_1,\dots,t_k)
            \, \mathbf{1}_{t_1 + \dots + t_k \leq 1 + \varepsilon} 
            \, dt_k
        \right)^2
        dt_1\cdots dt_{k-1}. 
    \end{align*}
    In particular, we write $I(F) := I_{0} (F)$ and $J(F) := J_{0} (F)$.
\end{defi}

\begin{defi}[Variational quotient]\label{variational_quotient}
    Fix an integer $k \geq 2$ and a real $0 \leq \varepsilon < 1$. The key value that we optimize is
    \begin{equation*}
        M_{k,\varepsilon} := \sup_{F}
        \frac{k  J_{\varepsilon}(F)}{I_{\varepsilon}(F)} 
    \end{equation*}
    where the supremum is taken over all functions $F$ satisfying $I_{\varepsilon}(F) > 0$ that are locally square-integrable and symmetric. If $\varepsilon=0$, then we simply refer to this value as $M_k$.
\end{defi}

\begin{defi}[Admissible tuples]
    For any integer $k \geq 1$, an admissible $k$-tuple
    is a sequence 
    \begin{equation*}
        \mathcal{H} = (h_1, \ldots ,h_k)
    \end{equation*} 
    of $k$ integers $h_1 < \cdots < h_k$ which avoids at least one residue class $a_p \pmod p$ for every prime $p$.
\end{defi}

\begin{claim}[Weak Dickson--Hardy--Littlewood conjecture]
    Let $1 \leq j \leq k$ be integers. For any admissible $k$-tuple $\mathcal{H} = (h_1, \dots, h_k)$ there exist infinitely many translates $n + \mathcal{H} := (n + h_1, \dots, n + h_k)$ of $\mathcal{H}$ which contain at least $j$ primes. We refer to this claim as $DHL[k,j]$.
\end{claim}

Our main starting point is \cite[Theorems 22 and 26]{PM8b}. Here we recall a combined and slightly specialized version of these two: 

\begin{thm}[Sieving on an epsilon-enlarged simplex]\label{PM_theorem}
    Let $k \ge 2$ and $m \ge 1$ be fixed integers, and let
    $0 \leq \varepsilon < 1$ be a fixed real. If
    \begin{equation*}
        M_{k,\varepsilon}> 4m,
    \end{equation*}
    then $\operatorname{DHL}[k,m+1]$ holds.
\end{thm}

\subsection*{Sources of improvements}

To summarize the discussion, we can describe the usual framework as follows:
\begin{equation*}
    \text{GPY set-up} \quad \longrightarrow \quad
    \text{dealing with the error term} \quad \longrightarrow \quad
    \text{optimization problem}.
\end{equation*}
Historically, since the set-up itself was used for the first time, researchers have been finding improvements in all of the links of the above chain:
\begin{enumerate}
    \item 
        We can expand the sieve support by deploying stronger distributional claims than the Bombieri-Vinogradov theorem. This is the strategy Zhang followed to get the first result showing that $H_1$ is indeed finite \cite{Zhang}. The key innovations of Stadlmann may also belong to this category \cite{JuliaS_Zhang_paper, JuliaS}.
    \item
        Instead of focusing on BV-type claims, we can expand the sieve support by using various arithmetic tricks. One example of this sort is the $\varepsilon$-trick, first described in \cite{PM8b}.
    \item 
        One can instead modify the GPY framework itself. A prominent example is Maynard's key idea of considering multidimensional Selberg weights. 
        
        Another approach of this kind is to replace the prime indicator function $\mathbf{1}_{\mathbf{P}}$ with a suitable minorant \cite{BakerIrving}.
    \item 
        We can focus on the resulting optimization problems. In \cite{PM8b}, a substantial effort was put into numerical calculations involving $M_k$ and $M_{k, \varepsilon}$ to obtain $H_1\leq 246$ unconditionally and $H_1 \leq 6$ assuming the generalized Elliott-Halberstam conjecture. 
        
        Another example in this category is searching for shorter admissible $k$-tuples. While this part is already optimized in the most interesting $m=1$ case, there is still some room for improvement for $m \geq 2$.
\end{enumerate}
The approach of this paper sits firmly in the fourth category: the key motivation for this project was to study the variational quotient $M_k$ (as well as $M_{k, \varepsilon}$) for middle-sized $k \approx 50$, where some interesting novelties concerning $H_1$ may possibly lie. The key presumption was that these quotients can be estimated in a more efficient way than using high-degree symmetric polynomials, as in \cite{PM8b, JuliaS}.

In this project, several families of functions were tested. An interesting finding was that the simple choice of the sieving function
\begin{equation}\label{ansatz_1}
    F (t_1, \dots, t_k) =  
    \frac{Q(1 + \varepsilon - P_1)}{\prod_{i=1}^k (1 + at_i)},
\end{equation}
where $a > 0$, $P_1 := \sum_{i=1}^k t_i$, and $Q$ is a non-zero polynomial, works exceedingly well with Theorem \ref{PM_theorem} and can already match the bounds for $M_k$ obtained in Polymath8b.

\begin{table}[ht]
    \centering
    {\sffamily\small
    \setlength{\tabcolsep}{14pt}
        \begin{tabular}{cccc}
            \toprule
            $k$ & $\deg(Q)=0$ & $\deg(Q) \leq 4$ & PM8b (Krylov) \\
            \midrule
              2 & 1.38419 & 1.38518 & 1.38593 \\
              3 & 1.64093 & 1.64494 & 1.64644 \\
              5 & 1.99435 & 2.00501 & 2.00714 \\
             10 & 2.52060 & 2.54252 & 2.54547 \\
             20 & 3.09075 & 3.12287 & 3.12756 \\
             50 & 3.89468 & 3.93707 & 3.93586 \\
             53 & 3.94731 & 3.99017 & 3.98621 \\
             54 & 3.96422 & 4.00723 & 4.00223 \\
            100 & 4.52955 & 4.57584 & 4.46424 \\
            \bottomrule
        \end{tabular}
    }
    \caption{Lower bounds on $M_k$ (the $\varepsilon = 0$ case) obtained by optimizing functions of the form (\ref{ansatz_1}) with $a = k \log k$. We also provide the Polymath8b results obtained by the Krylov subspace method and the polynomial basis.}
    \label{tab_Mk}
\end{table}

The form (\ref{ansatz_1}) is inspired by the choice Maynard picked to get $H_m \ll m^3 \exp(4m)$ \cite{Maynard}. However, there are two differences in our approach:
\begin{enumerate}
    \item 
        Even when $k$ is very large, we evaluate the key integrals $I_{\varepsilon}(F)$ and $J_{\varepsilon}(F)$ directly. The product form of $F$ lets us write these integrals in a convolutional form. Then, we can compute them by applying the Laplace or Fourier transform.
    \item
        The reciprocal product is enriched by the polynomial $Q$, which gives us access to a bit more structure and leads to stronger numerics later. One natural idea would also be to include an extra 'free' (not multiplied by the product) polynomial to (\ref{ansatz_1}), but in various experiments optimizers tended to discard it.
\end{enumerate}

At this point, extrapolating the results from Table \ref{tab_Mk} clearly suggests that functions of the form (\ref{ansatz_1}) can lead to strong lower bounds for $H_2$, $H_3$, etc. The detailed analysis of the result below is described in Section \ref{section_numerics}.

\begin{thm}\label{DHL_claims}
    The following claims hold:
    \begin{equation*}
        DHL[\MTwoBound, 3],     \quad
        DHL[\MThreeBound, 4],    \quad
        DHL[\MFourBound, 5],    \quad
        DHL[\MFiveBound, 6].
    \end{equation*}
    This implies
    \begin{equation*}
        H_2 \leq \HTwoBound, \quad 
        H_3 \leq \HThreeBound, \quad    % 2563692
        H_4 \leq \HFourBound, \quad
        H_5 \leq \HFiveBound.
    \end{equation*}
\end{thm}

We also prove the following asymptotic result, which explains the efficacy of the studied method.

\begin{thm}\label{asymptotyczny_teoremat}
    The following lower bound holds for $k$ sufficiently large:
    \[  M_k \geq
    \left( \log k - 
        2 \left( \gamma + \frac{C_1}{C_0} \right)
        + \frac{1}{\log k} \left( \gamma^2 + \frac{\pi^2}{6} + 2\gamma \frac{C_1}{C_0} + \frac{C_2}{C_0} \right) \right)
        \, + \, O \left( \frac{\log^4 k}{k}  \right), \]
    where $\gamma$ stands for the Euler-Mascheroni constant and, for any integer $m \geq 0$, the constants $C_m$ are defined as
    \begin{equation*}
        C_m := \frac{1}{2 \pi} \int \limits_{-\infty}^{\infty}
        e^{i \gamma \tau} (1 + i \tau)^{i \tau} \log^m (1 + i \tau)
        \, d \tau.
    \end{equation*}
    The integral above converges absolutely for all valid choices of $m$ and $C_0 > 0$.

    Moreover, we numerically obtain 
    \begin{equation*}
        M_k \geq  \log k - 0.273572 + o(1)
    \end{equation*}
\end{thm}

We conclude that the lower bound is only $\approx 0.274$ away from the (not sharp) upper bound for $M_k$ provided in \cite[Corollary 37]{PM8b} as $k \rightarrow \infty$.

\subsection*{AI declaration}

The author of this paper was assisted by LLMs. The most notable examples of such assistance are:
\begin{itemize}
    \item 
        The idea to exploit the convolutional structure of the key functionals via the Laplace transform was suggested by GPT 5.6 Sol at an early stage of this project when the author was trying to develop a more numerically stable optimizer for a wider family of Polymath8b-related problems.
    \item
        Several tested optimizers were developed with significant assistance from GPT 5.6 Sol/GPT 6 and Claude Opus 5/5.5.
    \item 
        Some useful technical identities, for example (\ref{tricky_integral_formula}) or (\ref{from_MT_to_Gamma}), were proposed by GPT 5.6 Sol/GPT 6.
    \item
        Checking mathematical arguments for potential errors and general proofreading.
\end{itemize}

\clearpage

\section{Initial reduction}\label{section_laplace}

In this section we show that the integrals $I_\varepsilon (F)$ and $J_\varepsilon (F)$ with $F$ of the form (\ref{ansatz_1}) can be reduced to integrals of a single variable.

\begin{thm}\label{theorem_laplace_form}
    Fix an integer $k \geq 2$. Let $a \geq 0 $ and $0 \leq \varepsilon < 1$ be fixed real parameters. Let $Q$ be a fixed non-zero real polynomial. Define $P_1 := \sum_{i=1}^k t_i$ and consider
    \begin{equation*}
        F(t_1, \dots , t_k ) = 
        \frac{Q(1 + \varepsilon - P_1)}{\prod_{i=1}^k (1 + at_i)},
    \end{equation*}
    Then, for any fixed $\sigma>0$, we have
    \begin{align}
        I_\varepsilon (F) &= \frac{1}{2 \pi i}
        \int \limits_{\sigma - i \infty}^{\sigma + i \infty}
        K_{Q, \varepsilon}(s) \,
        W (s; a)^k \, ds, \label{I(F)_Laplace_form} 
        \\
        J_{\varepsilon} (F) &= \frac{1}{2\pi i}
        \int \limits_{\sigma-i\infty}^{\sigma+i\infty} 
        e^{(1 - \varepsilon)s} \,W(s;a)^{k-1} \,
        V_{Q,\varepsilon}(s; a)\,ds,  \label{J(F)_Laplace_form} 
    \end{align}
    where
    \begin{equation}\label{definition_W_V}
        \begin{aligned}
            W(s;a) &:= \int \limits_0^\infty 
            \frac{e^{-st} \, dt}{(1 + at)^2},
            \\
            V_{Q,\varepsilon} (s; a) &:=
            \int \limits_0^\infty 
            \left( \int \limits_0^{t + 2 \varepsilon}
            \frac{Q(t + 2 \varepsilon - y)}{1 + ay} \, dy \right)^2
            e^{-st} \, dt,
            \\
            K_{Q,\varepsilon}(s) &:=  Q \left( 1 + \varepsilon - \frac{d}{ds} \right)^2 
            \left( \frac{e^{(1 + \varepsilon)s} - 1}{s} \right).
        \end{aligned}
    \end{equation}
\end{thm}

\begin{proof}
First, note that the integrals (\ref{I(F)_Laplace_form}) and (\ref{J(F)_Laplace_form}) converge absolutely due to

\begin{equation*}
    W(\sigma + i \omega;a) \ll |\omega|^{-1}, \quad
    V_{Q, \varepsilon}(\sigma + i \omega;a) \ll |\omega|^{-1}, \quad
    \frac{d^m}{ds^m} 
    \left( \frac{e^{(1+\varepsilon)s} - 1}{s} \right)
    \ll_m |\omega|^{-1}.
\end{equation*}

Let us start with $I_\varepsilon (F)$. Put  $X=P_1$ and rearrange
\begin{equation*}
    I_{\varepsilon}(F) = \int \limits_0^{1+\varepsilon} 
    Q(1 + \varepsilon - X)^2 \, C_k (X; a) \, dX,
\end{equation*}
where we define $w(t;a) := \frac{1}{(1+at)^2}$ and

\begin{equation*}
    C_k(X; a) := \int 
    \limits_{\substack{t_1, \dots , t_k \ge 0 \\ t_1+\cdots+t_k = X}}
    w(t_1;a) \cdots w(t_k;a) \, d \sigma = 
    w^{*k}(X;a).
\end{equation*}
In the above integral, we lifted the Lebesgue measure $dt_1 \cdots dt_{k-1}$ up to the hyperplane $t_1 + \dots + t_k = X$. We define the Laplace transform of $w$:

\begin{equation}\label{W_definition}
    W(s;a) := \mathcal{L} \left\{ w(\cdot; a) \right\} (s).
\end{equation}
Switching to the frequency domain allows us to turn a convolution into a tractable product:

\begin{equation}\label{C_k_conv_def}
    C_k(X; a) = \mathcal{L}^{-1} \left\{ W(s;a)^k \right\}(X) =
    \frac{1}{2 \pi i}
    \int \limits_{\sigma - i \infty}^{\sigma + i \infty}
    e^{sX} \, W (s; a)^k \, ds.
\end{equation}
By the Fubini-Tonelli theorem, it suffices to show that

\begin{equation}\label{tricky_integral_formula}
    \int \limits_0^{1 + \varepsilon} Q(1 + \varepsilon - X)^2 
    \, e^{sX} \, dX = 
    Q \left( 1 + \varepsilon - \frac{d}{ds} \right)^2 
    \left( \frac{e^{(1 + \varepsilon)s} - 1}{s} \right).
\end{equation}
Indeed, if we decompose $Q(1 + \varepsilon -X)^2$ into a weighted sum of monomials $X^n$ for $n \geq 0$, then we can work out (\ref{tricky_integral_formula}) term by term. By the Taylor expansion, we get

\begin{equation*}
        \int_0^{1+\varepsilon}X^ne^{sX}\,dX 
        = \sum_{m=0}^{\infty} 
        \frac{(1 + \varepsilon)^{m + n + 1}s^m}{m!(m+n+1)} 
        = \frac{d^n}{ds^n}\left(\frac{e^{(1+\varepsilon)s}-1}{s}\right).
\end{equation*}

For the $J_{\varepsilon}(F)$ part, we consider $X=t_1 + \cdots + t_{k-1}$ and obtain

\begin{equation*}
    J_{\varepsilon}(F) =
    \int \limits_0^{1 - \varepsilon} C_{k-1}(X;a)
    \left( \int \limits_0^{1 + \varepsilon - X}
    \frac{Q (1 + \varepsilon - X - t_k)}{1+a t_k} \, dt_k \right)^2
    \,dX = (C_{k-1} * v_{Q,\varepsilon}^2) (1 - \varepsilon) ,
\end{equation*}
where $v_{Q,\varepsilon}$ stands for the inner integral, that is

\begin{equation*}
    v_{Q,\varepsilon} (y;a) := 
    \int \limits_0^{y + 2 \varepsilon}
    \frac{Q(y + 2 \varepsilon - t)}{1+a t} \, dt .
\end{equation*}
Noting that $V_{Q,\varepsilon} (s; a) = 
\mathcal L\{v_{Q,\varepsilon} (\cdot; a)^2 \} (s)$ we reach the desired conclusion.

\end{proof}

% There is also a little technical lemma that allows us to discard certain functions when estimating tails. 

% \begin{lemma}[Majorization]
%     We apply the inequalities
%     \begin{equation*}
%         \left| \frac{W(s;a)}{W(\sigma ;a)} \right| \leq 1, \quad
%         \left| \frac{V(s;a)}{V(\sigma ;a)} \right| \leq 1, \quad
%         Q \left( 1 + \varepsilon - \frac{d}{ds} \right)^2 
%         \left( \frac{e^{(1 + \varepsilon)s} - 1}{s} \right) \leq
%         Q \left( 1 + \varepsilon - \frac{d}{ds} \right)^2 
%         \left( \frac{e^{(1 + \varepsilon) \sigma} - 1}{\sigma} \right).
%     \end{equation*}
% \end{lemma}

% \[
%     |\mathcal Lg(\sigma+i\omega)|\le\mathcal Lg(\sigma)
% \]

\subsection*{Convergence as $k \rightarrow \infty$}

In the following discussion, we study how the integrals described in Theorem \ref{theorem_laplace_form} behave as functions of the involved parameters. To simplify, we focus on the case $\varepsilon = 0$ and $Q = 1$. It is quite intuitive why this might be general enough. As $k$ grows, the optimal $\varepsilon$ usually sits very close to $0$ anyway. Also, the polynomial $Q$ is merely a correction term; the main contribution to any lower bound for $M_k$ in this paper comes from the product of reciprocals. 

We define
\begin{equation}\label{F_base_defined}
    V(s;a) := V_{1, 0}(s;a), \quad 
    K (s) := \frac{e^s - 1}{s}, \quad
    F_{\text{base}}(t_1, \dots, t_k) := \prod_{i=1}^k \frac{1}{1 + at_i}. 
\end{equation}
While technically speaking $F_{\text{base}}$ depends on $k$, we do not explicitly mention that in our notation for the sake of brevity. Let us normalize the resulting integrals to control their magnitudes as $k \rightarrow \infty$:

\begin{align}
    I(F_{\text{base}}) &= 
    \frac{K (\sigma) \, W(\sigma; a)^k}{2 \pi i}
    \int \limits_{\sigma - i \infty}^{\sigma + i \infty}
    \frac{K (s)}{K (\sigma) }  
    \left( \frac{W (s; a)}{W(\sigma; a)} \right)^k \, ds, 
    \label{norm_I(F)} \\
    J(F_{\text{base}}) &= 
    \frac{e^{\sigma} \, W(\sigma, a)^{k-1} V(\sigma, a)}{2 \pi i}
    \int \limits_{\sigma - i \infty}^{\sigma + i \infty}
    e^{s - \sigma} 
    \frac{V (s; a)}{V(\sigma; a)} 
    \left( \frac{W (s; a)}{W(\sigma; a)} \right)^{k-1} 
    \, ds. \label{norm_J(F)}
\end{align}

We can now plot the integrands and experiment with various parameter choices, namely $a$, $\sigma$, and the scale itself. The conclusions are as follows.

\begin{itemize}
    \item 
        If we consider the magnitudes of these functions, their weights are visibly concentrated on the interval $|\operatorname{Im}(s)| \ll \log k$. The decay rate seems rapid, so bounding the tails should not be a serious obstacle.
    \item
        The integrated functions might oscillate too fast to allow good control. Among the choices tested, $\sigma=\log k$ looks natural enough and seems to almost fully stabilize the phase.
    \item
        Motivated by Maynard's original choice in \cite{Maynard}, we pick $a = k \log k$. When we calculate the lower bounds for $M_k$ numerically in Section \ref{section_numerics}, we find that this choice of $a$ is indeed close to optimal.    
    \item
        The magnitude appears to decrease along the vertical half-line $\operatorname{Im}(s)>0$. In fact, this property holds for $|W(s;a)|$ alone, as we prove in Lemma \ref{monotonicity_lemma}.
\end{itemize}

In Section \ref{section_asymptotic} we describe in detail the asymptotic limit of the integrals $I(F_{\text{base}})$ and $J(F_{\text{base}})$. As a consequence of this analysis, we deduce Theorem \ref{asymptotyczny_teoremat}.

\begin{figure}[h]
    \centering
    \includegraphics[width=14cm]{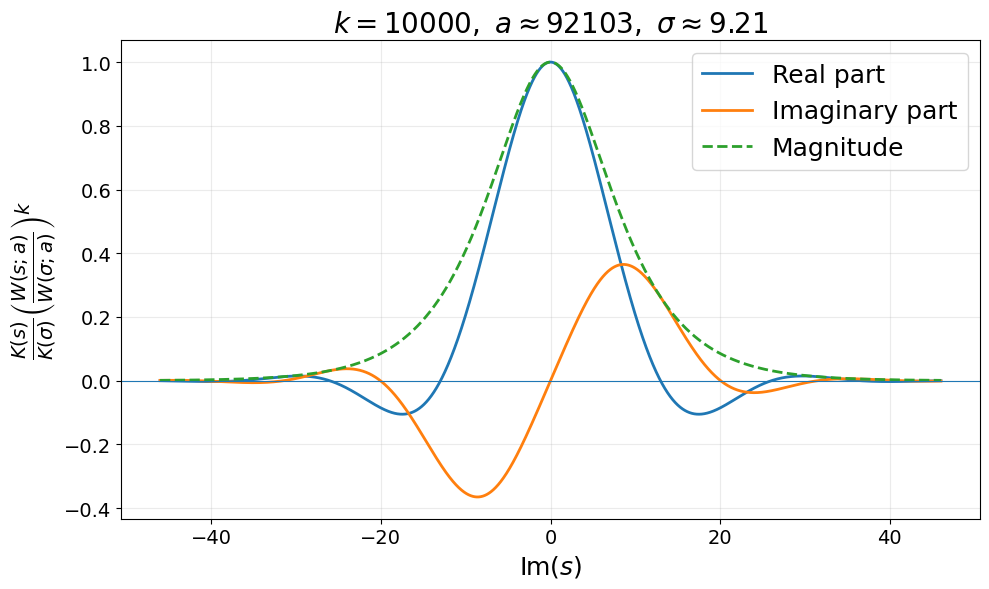}
    \caption{The integrand from the right-hand side of (\ref{norm_I(F)}) with $k=10\,000$, $a=k \log k$, and $\sigma = \log k$.}
    \label{integral_plot}
\end{figure}

\section{Properties of the $W$ function}

\begin{lemma}[Monotonicity]\label{monotonicity_lemma}
    Fix $a, \sigma_0 >0$. The real function $\omega \mapsto \left| W(\sigma_0 + i \omega ; a) \right|$ is even and strictly decreasing on the positive half-line.
\end{lemma}

\begin{proof}
    
 It follows from the complex conjugation that the function is even. To deal with monotonicity, we put $s=\sigma_0 + i \omega$ and integrate by parts:

 \begin{equation*}
     W(s;a) = \int \limits_0^{\infty}  \frac{e^{-st} \, dt}{(1 + at)^2} =
     \frac{1}{a} - \frac{s}{a} \int \limits_0^{\infty} 
     \frac{e^{-st} \, dt}{1 + at}.
 \end{equation*}
 We substitute $u=st$. Then, we can rotate the contour of integration to the real half-line, because there are no singularities on the way. Since $|\operatorname{Arg}(s)| < \pi / 2$, the connecting arc and the far tail are arbitrarily small if the radius is big enough. Hence, 

 \begin{equation*}
      W(s;a) = \frac{1}{a} - 
      \frac{s}{a} \int \limits_0^{\infty} 
     \frac{e^{-u} \, du}{s + au} = 
      \frac{s}{a} \int \limits_0^{\infty} 
      \left( \frac{1}{s}  - \frac{1}{s + au} \right) e^{-u} \, du =
      \int \limits_0^{\infty}  \frac{u e^{-u} \, du}{s + au}
 \end{equation*}
We need to separate the real and the imaginary parts of the denominator. We put $r_u := \sigma_0 + au$ and get

\begin{equation*}
    W(\sigma_0 + i \omega;a) =
    \int \limits_0^{\infty}  \frac{u e^{-u} \, du}{r_u + i \omega} =
    \int \limits_0^{\infty}  
    \frac{u e^{-u} (r_u - i \omega)}{r_u^2 + \omega^2} \, du.
\end{equation*}
Taking the squared absolute values leads to

\begin{align*} 
    \left| W(\sigma_0 + i\omega;a) \right|^2 &=
    \int \limits_{0}^{\infty} \int \limits_{0}^{\infty}
    uv e^{-u-v} \frac{r_u r_v + \omega^2} 
    {(r_u^2 + \omega^2)(r_v^2 + \omega^2)} \, du \, dv \\[4pt] 
    &=
    \int \limits_{0}^{\infty} \int \limits_{0}^{\infty}
    \frac{uv e^{- u - v}}{r_u + r_v} 
    \left( \frac{r_u}{r_u^2+\omega^2} +\frac{r_v}{r_v^2 + \omega^2} \right) 
    du \, dv,
\end{align*} 
where in the first line we exploited the symmetry between $u$ and $v$ to discard $ \pm i \omega (r_u - r_v)$ from the numerator. Consequently, by the dominated convergence theorem, for every $\omega > 0$ we have

\begin{equation*}
    \frac{d}{d \omega} \left| W(\sigma_0 + i\omega;a) \right|^2
    = -2 \omega
    \int \limits_{0}^{\infty} \int \limits_{0}^{\infty}
    \frac{uv e^{-u-v}}{r_u + r_v} \left( \frac{r_u}{(r_u^2 + \omega^2)^2} + \frac{r_v}{(r_v^2 + \omega^2)^2} \right) du \, dv <0. 
\end{equation*}

\end{proof}

\begin{lemma}\label{properties_of_W}
    Let $z$ be a complex number with $z_0 := \operatorname{Re}(z)>0$. The function $W$, as defined in (\ref{W_definition}), has the following properties.
    \begin{enumerate}    
        \item (The scaling identity) For any real $d>0$, we have
            \begin{equation}\label{scaling_W}
                W(z;d) = \frac{1}{d} \, W \left( \frac{z}{d};1 \right)
            \end{equation}
        \item 
            (Explicit form) We have
            \begin{equation}\label{explicit_form_W}
                W(z;1) = 1 - z e^z E_1(z),
            \end{equation}
            where $E_1(z) := \int_{z}^{\infty} \frac{e^{-t}}{t} \, dt$. 
        \item
            (Tail bound) Suppose additionally that $z_0 \leq 1/20$, then
            \begin{equation}\label{tail_bound_W}
                \left| \frac{W(z;1)}{W(z_0; 1)} \right| \leq
                \frac{5}{2 |z |}.
            \end{equation}  
    \end{enumerate}
    % \begin{equation*}
    %     \left| \frac{W(s;a)}{W(\sigma; a)} \right| \leq
    %     \frac{2k}{0.99 |\tau |}
    % \end{equation*}
\end{lemma}

\begin{proof}
    The scaling property follows directly from the definition. To get the explicit form, we integrate by parts to get
    \begin{equation*}
        W(z;1) = 1 - z \int \limits_0^{\infty} \frac{e^{-zt}}{1+t} \, dt.
    \end{equation*}
    It now suffices to substitute $z(1+t) \mapsto u$.
    
    For the tail bound, we integrate by parts in the other direction and get
    
    \begin{equation}\label{lemma_upper_bound_W}
        W(z;1) = \frac{1}{z} 
        - \frac{2}{z} \int_0^{\infty} \frac{e^{-zt}}{(1+t)^3} \, dt 
        \quad \Longrightarrow \quad 
        |W(z;1)| \leq \frac{2}{|z|}.
    \end{equation}
    Now we only need a decent lower bound on $|W(z_0; 1)|$. We use the explicit form (\ref{explicit_form_W}) together with
    \begin{equation*}
        E_1(z_0) =
        \left( \int \limits_{z_0}^1 + \int \limits_1^{\infty} \right) 
        \frac{e^{-t}}{t} \, dt \leq \log \left( z_0^{-1} \right) + e^{-1}.
    \end{equation*}
    This leads to
    \begin{equation*}
        |W(z_0;1)| \geq 1 - z_0e^{z_0} 
        \left( \log \left( z_0^{-1} \right) + e^{-1} \right).
    \end{equation*}
    The right-hand side treated as a real function of $z_0$ has a negative derivative on the relevant interval $(0, 1/20]$, so we conclude that
    \begin{equation*}
        |W(z_0;1)| \geq 1 - \frac{1}{20} e^{1/20} \left( \log(20) + e^{-1} \right) > 0.8.
    \end{equation*}
    Together with (\ref{lemma_upper_bound_W}), this concludes the proof.

\end{proof}

\section{Numerical calculations}\label{section_numerics}

In this section we apply Theorems \ref{PM_theorem} and \ref{theorem_laplace_form}, and numerically calculate the integrals $I_\varepsilon(F)$ and $J_\varepsilon (F)$ in the forms given by (\ref{I(F)_Laplace_form}--\ref{J(F)_Laplace_form}) to get explicit bounds for $H_m$ for $m = 2, 3, 4, 5$ (and possibly beyond). The process is described below.

\begin{itemize}
    \item
        The process has two parts: optimization and certification. In the first step, we find a candidate for $Q$. Next, we rigorously compute the relevant quantities for the chosen candidate.
    \item 
        We normalize the powers of $W(s;a)$ by the corresponding powers of $W(\sigma;a)$ to control their magnitudes as $k \rightarrow \infty$, just like in (\ref{norm_I(F)}--\ref{norm_J(F)}). 
    \item
        We choose $a = \alpha k \log k$ for some $\alpha > 0$. Usually, we take $0.97 < \alpha < 1.00$. We also fix $\sigma = \alpha \log k$ and rescale $s = \sigma( 1 + i \tau )$ for $\tau \in \mathbf{R}$. The parameters $k$, $a$, $\varepsilon$, and $\alpha$ are all fixed.
    \item 
        To get a good candidate $Q$ we optimize over polynomials $Q$ satisfying $\operatorname{deg}(Q) \leq 8$, i.e.:
        \begin{equation*}
            Q(X)=\sum_{i=0}^{8} c_i X^i, 
            \quad \mathbf{c}=(c_0,\ldots,c_8)^T
        \end{equation*}
        The coefficients $c_i$ are free parameters, which leads to a generalized Rayleigh quotient problem of the form:
        \begin{equation*}
            \lambda_{\max} =\max_{\mathbf c\ne0} 
            \frac{\mathbf c^TM_J\mathbf c}{\mathbf c^TM_I\mathbf c},
        \end{equation*}
        where the $(i,j)$-th cell of $M_I$ is the value of $I_\varepsilon (F)$ integral but with the polynomial $Q(1 + \varepsilon - X)^2$ replaced with a monomial $(1 + \varepsilon - X)^{i+j}$. 
        
        For $J_\varepsilon (F)$ and $M_J$, it is a bit more complicated. To get the $(i,j)$-th cell, we need to replace the entire squared inner integral of $J_\varepsilon (F)$ with

        \begin{equation*}
            \left(\int \limits_0^{t+2\varepsilon} 
            \frac{(t + 2 \varepsilon - y)^i}{1 + ay} \, dy \right) \left(\int \limits_0^{t + 2 \varepsilon} 
            \frac{(t + 2\varepsilon - y)^j}{1 + ay} \, dy \right)
        \end{equation*} 
        and multiply it by $k$ to account for having $k$ functionals in the original problem of maximizing the quotient from Definition \ref{variational_quotient}.
        
        This way, we end up with a problem of calculating the cells of two $9 \times 9$ matrices, and then we solve the eigenvalue problem:
        \begin{equation*}
            M_J \mathbf{c} = \lambda M_I \mathbf{c}.
        \end{equation*}
    \item 
        As Fig. \ref{integral_plot} suggests, the contribution of the involved integrals is concentrated near the origin. To approximate the matrices $M_I$ and $M_J$ we set a parameter $T$, and we apply the Gauss-Legendre quadrature on the interval $\tau \in [-T,T]$. We usually set $T$ at $40$. 
    \item 
        Once we select a candidate $Q$, we calculate the integrals $I_\varepsilon(F)$ and $J_\varepsilon(F)$ directly. Over the central interval $[-T,T]$, we use the Arb library (Python), which implements the Gauss–Legendre quadrature with explicit error bounds.
    \item
        The middle part $|\tau| \in [T,4k]$ describes the region where it is too early for the tail bound from Lemma \ref{properties_of_W} to be effective. Instead, we apply Lemma \ref{monotonicity_lemma} with the value of the quotient $|W(s;a)/W(\sigma;a)|$ being evaluated at $\tau= \pm T$. 

        For this step, the functions $V_{Q,\varepsilon}(s;a)$, $K_{Q,\varepsilon}(s)$ are bounded as
        \begin{equation}\label{Laplace_real_bound}
            \left| V_{Q,\varepsilon}(s;a) \right| 
            \leq V_{Q,\varepsilon}(\sigma ;a), \quad
            \left| K_{Q,\varepsilon}(s) \right| \leq 
            K_{Q,\varepsilon}(\sigma),
        \end{equation}
        since the first one can be represented as the Laplace transform of a positive function and the other one has the integral form (\ref{tricky_integral_formula}). In both cases, the phase can be easily discarded by the triangle inequality.
    \item 
        We apply Lemma \ref{properties_of_W} and (\ref{Laplace_real_bound}) to bound the tails $|\tau| > 4k$.
    
\end{itemize}

The following results imply new concrete bounds for $H_2$, $H_3$, $H_4$, and $H_5$. The chosen polynomials can be found in the Jupyter notebook attached to the paper.

\begin{table}[H]
    \centering
    {\sffamily\small
    \setlength{\tabcolsep}{14pt}
        \begin{tabular}{cccc}
            \toprule
            $k$ & $\alpha$ & $\varepsilon$ & 
            Lower bound on $M_{k,\varepsilon}$ \\
            \midrule
              \MTwoBound   & 0.984   & 0.004    &  8.00017902  \\
              \MThreeBound & 0.980   & 0.0001   & 12.00000472  \\
              \MFourBound  & 0.985   & 0.00001  & 16.00000823  \\
              \MFiveBound  & 0.985   & 0.00001  & 20.00002452  \\
            \bottomrule
        \end{tabular}
    }
    \caption{Selected lower bounds for $M_{k,\varepsilon}$.}
    % \label{tab:Mk}
\end{table}

These results prove Theorem \ref{DHL_claims}. To get concrete upper bounds on $H_m$ for $m=2,3,4,5$, we require reasonable upper bounds on $H(k)$, the minimum diameter of an admissible $k$-tuple. We apply the shifted greedy sieve for $m=2,3$, the shifted Schinzel sieve for $m=4$, and the simple \textit{$k$ primes past $k$} construction for $m=5$ (see \cite{PM8b, Schinzel} for descriptions of these techniques). The results are:

% For $m=2,3$ we apply the shifted greedy sieve, and for $m=4$, the shifted Schinzel sieve \cite{PM8b, Schinzel}. The results are:

\begin{equation}\label{from_Mk_to_Hm}
    \begin{aligned}
        H_2 &\leq H(\MTwoBound)   \leq \HTwoBound, \\
        H_3 &\leq H(\MThreeBound) \leq \HThreeBound, \\
        H_4 &\leq H(\MFourBound)  \leq \HFourBound, \\
        H_5 &\leq H(\MFiveBound)  \leq \HFiveBound.
    \end{aligned}
\end{equation}
The bounds on $H(k)$ used here can certainly be improved by deeper calculations.

Below we demonstrate the comparison between the results obtained by this method and the upper bound predicted and proved in \cite{PM8b}:

\begin{figure}[h]
    \centering
    \includegraphics[width=15cm]{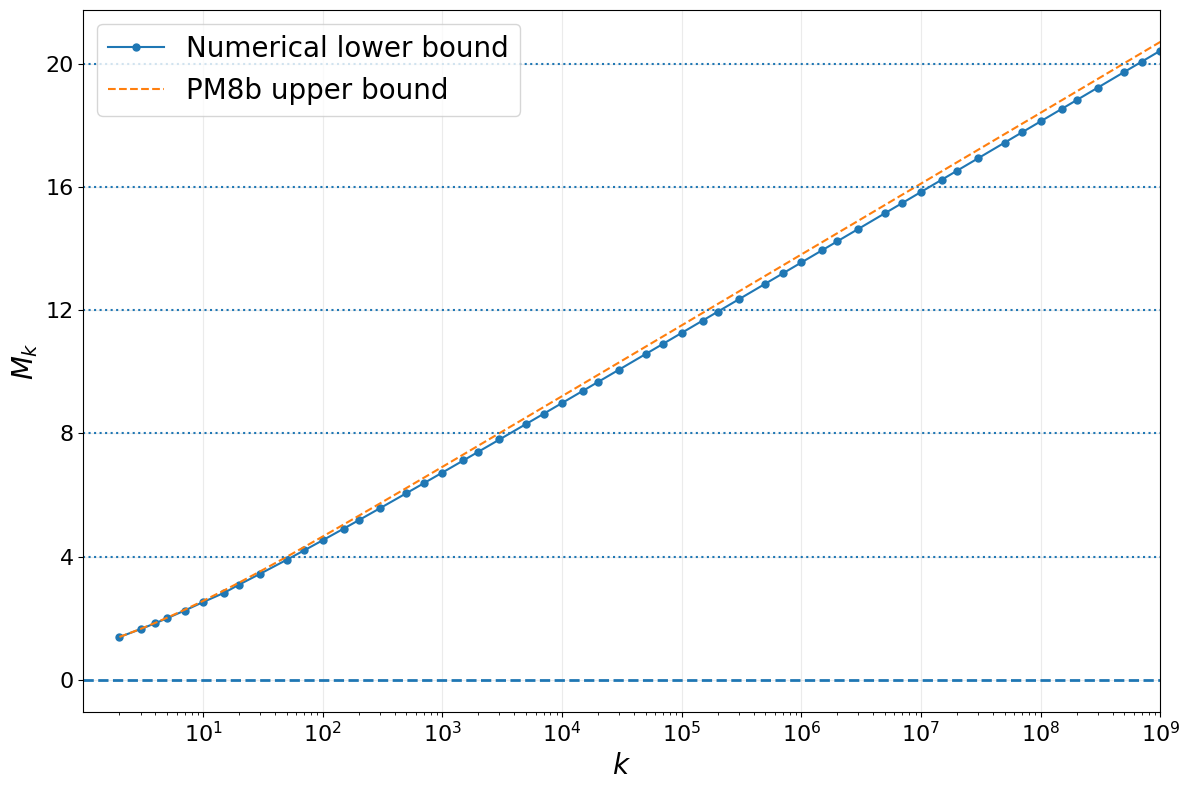}
    \caption{The lower bounds for $M_k$ obtained numerically for various $k$. In all these cases we take $\operatorname{deg}(Q)=0$, $\varepsilon = 0$, $a = k \log k$. For $k < 20$ we certify the results differently, but since this does not lead to any novel result, we skip the details.}
    \label{numerical_calculations_fig}
\end{figure}

To get the results from Fig. \ref{numerical_calculations_fig} we calculated the integrals $I_\varepsilon (F)$ and $J_\varepsilon (F)$ for
\begin{equation*}
    k \in \{ 2, 3, 4, 5, 7, 10 \} \cup \bigcup_{\ell=1}^8 
    \{ 1.5 \cdot 10^\ell, 2 \cdot 10^\ell, 
    3 \cdot 10^\ell , 5 \cdot 10^\ell, 7 \cdot 10^\ell, 
    10^{\ell + 1} \}.
\end{equation*}
The distance between the obtained $M_k$ and the upper bound seems to stall around $0.30$. This is the major clue leading to Theorem \ref{asymptotyczny_teoremat}.

% k^2 L^2 V(\sigma;a)

\section{Asymptotic behaviour}\label{section_asymptotic}

% \begin{thm}
%     State the Main Theorem (probably goes to the intro).
%     \[  M_k \ge 
%     \frac{k \sigma V(\sigma ; a) }{ W(\sigma ; a) } =
%     \left( \log k  - 2\gamma + \frac{\gamma^2 + \pi^2 / 6 - 2/k}{\log k} \right) \frac{Z_\infty-\mathcal E_J(k)} {Z_\infty+\mathcal E_I(k)}. \]
% \end{thm}

In this section we prove Theorem \ref{asymptotyczny_teoremat}. To this end, we apply Theorem \ref{theorem_laplace_form} with 

\begin{equation}\label{asymptotic_params_setting}
    Q = 1, \quad \varepsilon = 0, \quad
    \sigma= \log k, \quad a = k \log k,
\end{equation}
which is the same setting as in Subsection 'Convergence as $k \rightarrow \infty$' of Section \ref{section_laplace}. Since $Q=1$, we refer to $F$ as $F_{\text{base}}$. In this section we also put
\begin{equation*}
    L := \log k
\end{equation*}
for the sake of brevity. The main part of the proof is calculating the main terms for $I(F_{\text{base}})$ and $J(F_{\text{base}})$. The error terms can be grouped together and handled at once. 

We define the following replacements for $I$ and $J$ to avoid carrying the normalizing constants all the time:
\begin{equation}\label{defi_ZI_ZJ}
    \begin{aligned}
        Z_{I,k} &:= 
        \frac{1}{2 \pi}
        \int \limits_{-\infty}^{\infty}
        \frac{e^{i \sigma \tau} - e^{-\sigma}}{1 + i \tau}  
        \left( \frac{W (\sigma(1 + i \tau); a)}{W (\sigma; a)} \right)^k 
        d \tau, \\
         Z_{J,k} &:= \frac{\sigma^3 k^2}{2\pi }
         \int \limits_{-\infty}^{\infty} 
         e^{i \sigma \tau} 
         \left( \frac{W (\sigma(1 + i \tau); a)}{W (\sigma; a)} \right)^{k-1}
          V(\sigma(1 + i \tau); a) \, 
          d \tau.
    \end{aligned}
\end{equation}
It follows directly from (\ref{norm_I(F)}) and (\ref{norm_J(F)}) that we have

\begin{equation}\label{from_IJ_to_Z}
    \begin{aligned}
         I(F_{\mathrm{base}}) = 
         e^{\sigma} W(\sigma; a)^k \, Z_{I,k}, \quad
        J(F_{\mathrm{base}}) = 
        \frac{e^{\sigma} W(\sigma; a)^{k-1} }{\sigma^2 k^2} \, Z_{J,k}.
    \end{aligned}
\end{equation}

\begin{lemma}\label{Z_Lemma!}
    Let $k$ be a sufficiently large integer. Recall the notation of Theorem \ref{theorem_laplace_form} and the set of parameters given by (\ref{asymptotic_params_setting}). We have

    \begin{equation*}
        \begin{aligned}
            Z_{I,k} &= Z_\infty + O(L^3 / k), \\
            Z_{J,k} &= Z_\infty 
            \left( \left( L^2 - 
            2 \left( \gamma + \frac{C_1}{C_0} \right) L
            + \left( \gamma^2 + \frac{\pi^2}{6} + 2\gamma \frac{C_1}{C_0} + \frac{C_2}{C_0} \right) \right)
            \right) + \, O( L^4 / k),
        \end{aligned}
    \end{equation*}
    where the constants $C_m$ are defined in Theorem \ref{asymptotyczny_teoremat} and $Z_\infty := C_0$.
\end{lemma}

\begin{proof}

We apply the following decomposition:

\begin{equation*}
    \begin{aligned}
        Z_{I,k} &=  
        \mathrm{MT}_I + 
        \mathcal{E}_{I, \mathrm{middle}} + 
        \mathcal{E}_{I, \mathrm{tail}}, \\
        Z_{J,k} &=  
        \mathrm{MT}_J + 
        \mathcal{E}_{J, \mathrm{middle}} + 
        \mathcal{E}_{J, \mathrm{tail}},
    \end{aligned}
\end{equation*}
where for $T = 100 \, L$ we split $Z_{I,k}$ as

\begin{equation*}
    \mathrm{MT}_I := \int \limits_{|\tau| \leq T},  
    \quad \quad
    \mathcal{E}_{I, \mathrm{middle}} := \int \limits_{T \leq |\tau| \leq 4k}, 
    \quad \quad
    \mathcal{E}_{I, \mathrm{tail}} := \int \limits_{|\tau| \geq 4k}.
\end{equation*}
For $J(F_{\mathrm{base}})$ we consider an analogous decomposition.

\subsubsection*{Main term of $Z_{I,k}$}

We introduce $z := (1 + i \tau)/k$ for some $\tau$ satisfying $|\tau| \leq 100 \, L$. We also define $z_0 := \operatorname{Re}(z) = 1/k$. Recall the explicit form of $W$ from Lemma \ref{properties_of_W}. We apply the expansion

\begin{equation*}
    E_1(z) = -\gamma - \log z - \sum_{n=1}^\infty \frac{(-z)^n}{n \cdot n!}
\end{equation*}
valid on the $\operatorname{Re}(z)>0$ half-plane. For the rest of the argument, $\log z$ always denotes the principal branch of the complex logarithm. Extracting the leading terms, we reach the form:

\begin{equation*}
    W(z; 1) = 1 + z ( \gamma + \log z) + O \left( |z|^2 ( 1 + |\log z|) \right)
\end{equation*}
in some open neighborhood of $z=0$. From the Taylor approximation we obtain

\begin{equation}\label{logW(z;1)}
    \begin{aligned}
        \log( W(z;1)) &= W(z;1) - 1 + O \left( |z|^2 ( 1 + |\log z|)^2 \right) \\
        &= z ( \gamma + \log z) + O \left( |z|^2 ( 1 + |\log z|)^2 \right).
    \end{aligned}
\end{equation}
From (\ref{logW(z;1)}) considered for both $z$ and $z_0$ we have

\begin{equation}\label{log_W_summary}
    k \log \left( \frac{W(z;1)}{W(z_0;1)} \right) = 
    i \tau (\gamma - L) + (1 + i \tau) \log ( 1 + i \tau) + \Delta_k (\tau),
\end{equation}
where 

\begin{equation}\label{Delta_k_bound}
    \left| \Delta_{k} ( \tau ) \right| \ll \frac{L^2 (1+\tau^2)}{k}. 
\end{equation}

Taking the exponentials of (\ref{log_W_summary}), multiplying by $e^{i \sigma \tau}/(1 + i \tau)$, recalling that $a:= k L$, and applying Lemma \ref{properties_of_W} we get the following simplification of the  integrand of $Z_{I,k}$: 

\begin{equation}\label{Z_I_resolution}
    \begin{aligned}
        \frac{e^{i \sigma \tau}}{(1 + i \tau)} \left( \frac{W(\sigma (1 + i \tau); a)}{W(\sigma ; a)} \right)^k &=
        e^{i \gamma \tau} (1+ i \tau)^{i \tau} e^{\Delta_k (\tau)}, 
        \\
        \frac{e^{-\sigma}}{(1 + i \tau)} \left( \frac{W(\sigma (1 + i \tau); a)}{W(\sigma ; a)} \right)^k &\ll \frac{1}{k}.
    \end{aligned}
\end{equation}
The main term of (\ref{Z_I_resolution}) can be bounded as

\begin{equation}\label{bounding_limit_arctan}
    \left| e^{i \gamma \tau } (1+ i \tau)^{i \tau} \right|
    \leq e^{- |\tau| \arctan |\tau |},
\end{equation}
where we used the identity:

\begin{equation*}
    (1 + i \tau) = \sqrt{1 + \tau^2} \, e^{i \arctan \tau}.
\end{equation*}
Consequently, we get the following crude bound:

\begin{equation}\label{resolvent_crude_bound}
     \int \limits_0^T
     \left| e^{i \gamma \tau} (1+ i \tau)^{i \tau} \right|
     \left| e^{\Delta_k (\tau)} - 1 \right|  d \tau
     \ll  \frac{L^2}{k} .
\end{equation}
In particular, the contribution of the $\Delta_k (\tau)$ factor of (\ref{Z_I_resolution}) to the integral $Z_{I,k}$ can be bounded as

\begin{equation}\label{Delta_res}
    \frac{1}{\pi} \int \limits_0^T e^{-\tau \arctan \tau} 
    \left| e^{\Delta_k ( \tau ) } - 1\right| d \tau
    \ll 
    \frac{L^2}{k} \int \limits_0^T (1+\tau^2) \, 
    e^{-\tau \arctan \tau} d \tau \ll \frac{L^2}{k}.
\end{equation}

It now suffices to bound the tails of $Z_\infty$. We have the following simple inequality 

\begin{equation}\label{arctan_bound}
    - \arctan \tau \leq \frac{1}{\tau} - \frac{\pi}{2} \quad \text{for $\tau > 0$}.
\end{equation}
This gives

\begin{equation}\label{final_arctan_bound}
\frac{1}{\pi} 
\int \limits_T^{\infty} e^{- \tau \arctan \tau} \, d \tau 
\ll \int \limits_T^{\infty} e^{ - \pi \tau / 2} \, d \tau \ll
\exp \left( - \frac{\pi T}{2} \right) \ll k^{-150}.
\end{equation}

\subsubsection*{Main term of $Z_{J,k}$}

Recall the definition of $V(s;a)$ and $s = \sigma (1 + i \tau)$. Substitute $\sigma t \mapsto u$ within the integral. This time we work with $z:= (1 + i \tau)$. We get

\begin{equation}\label{log(1+ku)}
    \sigma^3 k^2 V (\sigma z; \sigma k) \, = \,
    \int \limits_0^{\infty} e^{-zu} \, \log^2  ( 1 + ku) \, du 
    \, = \, \operatorname{MT}_V(z) + R_V(z),
    % \int \limits_0^{\infty} e^{-zu} \, \log^2  ( 1 + ku) \, du 
    % = \int \limits_0^{\infty} e^{-zu} \, (L + \log u)^2 \, du 
    % \, + \, R_V (z),
\end{equation}
where

\begin{equation*}
    \begin{aligned}
        \operatorname{MT}_V(z) &:= \int \limits_0^{\infty} e^{-zu} 
        \, (L + \log u)^2 \, du, \\
        R_V (z) &:= \int \limits_0^{\infty} e^{-zu} 
        \left( \log^2  ( 1 + ku) - \log^2  (ku) \right) du.
    \end{aligned}
\end{equation*}

Let us deal with the main term of (\ref{log(1+ku)}) first. We have

\begin{equation*}
    \Gamma (w) := \int \limits_0^\infty e^{-v} v^{w-1} \, dv 
    \quad \Longrightarrow \quad \frac{\Gamma (w)}{z^w} = 
    \int \limits_0^\infty e^{-zu} u^{w-1} \, du,
\end{equation*}
which implies

\begin{equation}\label{from_MT_to_Gamma}
    \operatorname{MT}_V(z) =
    \left. \left( L + \frac{\partial}{\partial w} \right)^2 
    \frac{\Gamma (w)}{z^w} \right|_{w=1} = \,
    \frac{1}{z} \left( \left( L - \gamma - \log z \right)^2 + \frac{\pi^2}{6} \right),
\end{equation}
where to get the last equation, we simply expand everything and recall that 
\begin{equation*}
    \Gamma (1) = 1, \quad
    \Gamma' (1) = -\gamma, \quad 
    \Gamma'' (1) = \gamma^2 + \frac{\pi^2}{6}.
\end{equation*}

To deal with the error term $R_V (z)$, we substitute $ku \mapsto v$, which gives

\begin{equation}\label{RVz_upper_bound}
    \begin{aligned} |R_V(z)| &\le \frac1k\int_0^\infty e^{-v/k} \left|\log^2(1+v)-\log^2v\right|  dv \\ &\ll \frac{1+L^2}{k} \ll \frac{L^2}{k}. \end{aligned}
\end{equation} 

From (\ref{from_MT_to_Gamma}--\ref{RVz_upper_bound}) we get

\begin{equation}\label{resolving_Z_J,k_1st}
    \operatorname{MT}_J = \frac{1}{2 \pi} \int \limits_{-T}^{T}
    \frac{e^{i\sigma \tau}}{1 + i \tau} 
    \left( \frac{W(\sigma(1 + i \tau);a)}{W(\sigma; a)} \right)^{k-1}
    \left( (L - \gamma - \log z)^2 + \frac{\pi^2}{6} \right) d \tau
    + O \left( \frac{L^4}{k} \right).
\end{equation}
We apply (\ref{Z_I_resolution}), (\ref{resolvent_crude_bound}) to (\ref{resolving_Z_J,k_1st}) and conclude that

\begin{equation}\label{Z_jk_constants_form}
    Z_{J,k} = \frac{1}{2 \pi} \int \limits_{-T}^{T}
    e^{i \gamma \tau} (1+ i \tau)^{i \tau}
    \left( (L - \gamma - \log z)^2 + \frac{\pi^2}{6} \right) d \tau 
    \, + O \left( \frac{L^4}{k}  \right).
\end{equation}
The integral on the right-hand side of the equation above is absolutely convergent as $T \rightarrow \infty$, and its tail vanishes rapidly due to (\ref{bounding_limit_arctan}) and (\ref{final_arctan_bound}). In the statement of Theorem \ref{asymptotyczny_teoremat} we defined constants

\begin{equation*}
    C_m := \frac{1}{2 \pi} \int \limits_{-\infty}^{\infty}
    e^{i \gamma \tau} (1 + i \tau)^{i \tau} \log^m (1 + i \tau)
    \, d \tau.
\end{equation*}
In particular, $Z_\infty = C_0$. Also, the exponential decay (\ref{bounding_limit_arctan}) allows us to numerically calculate:

\begin{equation}
    C_0 = 0.216683\dots
\end{equation}
From (\ref{Z_jk_constants_form}), we calculate

\begin{equation}
    \begin{aligned}
        Z_{J,k} &=  C_0 \left((L - \gamma)^2 + \frac{\pi^2}{6} \right)
        - 2(L - \gamma) C_1 + C_2
        \, + \, O \left( \frac{L^4}{k}  \right) \\
        &= Z_\infty \left( L^2 - 
        2 \left( \gamma + \frac{C_1}{C_0} \right) L
        + \left( \gamma^2 + \frac{\pi^2}{6} + 2\gamma \frac{C_1}{C_0} + \frac{C_2}{C_0} \right) \right)
        \, + \, O \left( \frac{L^4}{k}  \right).
    \end{aligned}
\end{equation}

\subsubsection*{Error terms}

For this part, we need the inequalities
\begin{equation}\label{majorize_W_and_V}
    \left| \frac{W(s;a)}{W(\sigma ;a)} \right| \leq 1 , \quad
    \left| \frac{V(s;a)}{V(\sigma ;a)} \right| \leq 1,
\end{equation}
which follow directly from (\ref{definition_W_V}) and the triangle inequality. We also need an inequality dealing with the normalizing term used in (\ref{defi_ZI_ZJ}):
\begin{equation}\label{ineq_renormalize_J}
    \sigma^3 k^2 V(\sigma ; a) = \frac{1}{k} 
    V \left( \frac{1}{k} ; 1 \right) =
    \int \limits_0^\infty \log^2 (1 + ku) \, e^{-u} \, du \ll L^2.
\end{equation}
Now, we deal with the middle terms by an application of Lemma \ref{monotonicity_lemma} together with (\ref{Z_I_resolution}--\ref{bounding_limit_arctan}) and (\ref{majorize_W_and_V}--\ref{ineq_renormalize_J}):

\begin{equation}\label{middle_tail_est}
    \begin{aligned}
        \max \left\{ 
            |\mathcal{E}_{I, \mathrm{middle}}|, 
            |\mathcal{E}_{J, \mathrm{middle}}| 
        \right\} &\ll 
        L^2 
        \int \limits_{T}^{4k} 
        \left| \frac{W(\sigma (1 + i \tau);a)}{W(\sigma ;a)} \right|^{k-1} 
        \, d \tau  \\
        &\leq
        (4k-T)L^2 \left| \frac{W((1+iT)/k;1)}{W(1/k;1)} \right|^{k-1} \\
        &\ll k  L^3 \exp\left[ -\frac{k-1}{k} (T \arctan T)  
        + O \left( \frac{\log^4 k}{k} \right) \right] \\
        &\ll kL^3 \exp \left( - \pi T/2 \right) \\
        &\ll k^{-150}.
    \end{aligned}
\end{equation}

From Lemma \ref{properties_of_W} we get the following tail estimate:

\begin{equation*}
    \left| \frac{W(s;a)}{W(\sigma;a)} \right| \leq \frac{5k}{2|\tau|}.
\end{equation*}
This gives
\begin{equation}\label{far_tail_est}
    \max \left\{ |\mathcal{E}_{I, \mathrm{tail}}|, |\mathcal{E}_{J, \mathrm{tail}}| \right\} \ll 
    L^2 \left( \frac{5k}{2} \right)^{k-1}
    \int \limits_{4k}^{\infty} \frac{d \tau}{\tau^{k-1} } 
    \ll L^2 \left( \frac{5}{8} \right)^{k-1}.
\end{equation}

\end{proof}

\subsection*{Proof of Theorem \ref{asymptotyczny_teoremat}}

The proof of the main part follows directly from Lemmas \ref{properties_of_W} and \ref{Z_Lemma!}, because
\begin{equation*}
    M_k \geq \frac{k J(F_{\text{base}})}{I(F_{\text{base}})} =
    \frac{  Z_{J,k} }{ 
    \sigma W( 1/k; 1) \, Z_{I,k} }.
\end{equation*}
Using $ W (1/k;1) = 1 + O(L/k)$, we conclude

\begin{equation*}
    \begin{aligned} 
        M_k\ge {}& L - 2\left( \gamma + 
        \frac{C_1}{C_0}\right) + \frac1L\left( \gamma^2+\frac{\pi^2}{6} + 2 \gamma\frac{C_1}{C_0} + \frac{C_2}{C_0} \right) + O\left(\frac{L^4}{k}\right).
    \end{aligned}
\end{equation*}

\subsubsection*{The numerical part}

To obtain the stronger asymptotic constant, we consider the function family (\ref{ansatz_1}) with a general polynomial $Q$ and use numerical calculations. Put $c := 0.273572$. We apply Theorem \ref{theorem_laplace_form} and get:

\begin{equation}\label{defi_gen_ZI_ZJ}
    \begin{aligned}
        I(F) &= 
        \frac{\sigma W(\sigma ; a)^k}{2 \pi }
        \int \limits_{-\infty}^{\infty}
        \left( 
            \int \limits_0^1 \widetilde{Q}(\sigma(1 - X))^2 
            e^{\sigma (1 + i \tau) X} dX
        \right) 
        \left( \frac{W (\sigma(1 + i \tau); a)}{W (\sigma; a)} \right)^k 
        d \tau, \\
         J(F) &= \frac{\sigma e^\sigma W(\sigma ; a)^{k-1}}{2\pi }
         \int \limits_{-\infty}^{\infty} 
         e^{i \sigma \tau} 
         \left( \frac{W (\sigma(1 + i \tau); a)}{W (\sigma; a)} \right)^{k-1}
          V_Q (\sigma(1 + i \tau); a) \, 
          d \tau.
    \end{aligned}
\end{equation}
For technical reasons, it is better to work with the rescaled polynomial $\widetilde{Q}(\sigma x) := Q(x)$ defined for $x \geq 0$. Note that dependence on $k$ may sneak in; we keep $\widetilde{Q}$ fixed. We need to turn the factors dependent on $\widetilde{Q}$ into $k$-free expressions (up to possible small error terms). We rewrite

\begin{equation}\label{numer_kfree}
    \begin{aligned}
        \int \limits_0^1 \widetilde{Q}(\sigma(1 - X))^2 
        e^{\sigma (1 + i \tau) X} dX &= 
        \frac{e^{\sigma(1 + i \tau)}}{\sigma}
        \int \limits_0^L \widetilde{Q} (x)^2 \,
        e^{- (1 + i \tau)x} \, dx, \\
        V_Q (\sigma(1 + i \tau); a) &=
        \frac{1}{\sigma^3 k^2}
        \int \limits_0^\infty \left( 
        \int \limits_0^x \frac{\widetilde{Q}(x-u)}{u + \frac{1}{k} } du \right)^2
        e^{-(1+i \tau) x} \, dx.
    \end{aligned}
\end{equation}
Using the Taylor expansion, we resolve

\begin{equation}\label{numer_taylor}
    \begin{aligned}
        \int \limits_0^x \frac{\widetilde{Q}(x-u)}{u + \frac{1}{k} } du &=
        \sum_{j=0}^{d} \frac{(-1)^j \, \widetilde{Q}^{(j)} (x)}{j!}
        \int \limits_0^x \frac{u^j}{u + \frac{1}{k}} \, du \\
        &= \widetilde{Q} (x) ( \log x + L ) - (\mathcal{H}\widetilde{Q})(x) + o_{k \rightarrow \infty}(1),  
    \end{aligned}
\end{equation}
where $d := \operatorname{deg} (\widetilde{Q})$ and

\begin{equation*}
    (\mathcal{H} \widetilde{Q} )(x) := 
    \sum_{j=1}^{d} \frac{(-1)^{j+1} \, \widetilde{Q}^{(j)} (x) \, x^j}{j \cdot j!}.
\end{equation*}
Combining (\ref{defi_gen_ZI_ZJ}--\ref{numer_taylor}), we obtain

\begin{equation}\label{extracting_constant}
    k J(F) - (\log k) I(F) = 
    \frac{e^\sigma W(\sigma ; a)^{k-1}}{\pi k \sigma }
    \int \limits_{-\infty}^{\infty} 
    e^{i \sigma \tau} 
    \left( \frac{W (\sigma(1 + i \tau); a)}{W (\sigma; a)} \right)^{k-1}  \left( N ( \tau ) + O(L^{-1})\right) d \tau,  
\end{equation}
where we put

\begin{equation*}
    N(\tau ) :=
    \int \limits_0^\infty
    \widetilde{Q}(x) \left( \widetilde{Q}(x) \log x - 
    (\mathcal{H} \widetilde{Q}) (x) \right) e^{-(1 + i \tau)x} \, dx.
\end{equation*}
The error term $O(L^{-1})$ comes predominantly from squaring the integral from (\ref{numer_taylor}). We also need to account for the discrepancy between the leading terms via parts 1 and 2 of Lemma \ref{properties_of_W} and extending the upper integral from (\ref{numer_kfree}) to $\infty$, although these errors are much smaller.

By the triangle inequality, we have $|N(\tau)| \ll_{\widetilde{Q}} 1$. That means that we can safely apply the error estimates analogous to (\ref{middle_tail_est}--\ref{far_tail_est}) to the integral from (\ref{extracting_constant}). Therefore, from (\ref{Z_I_resolution}) it suffices to show that there exists a polynomial $\widetilde{Q}$ for which

% \begin{equation}
%      k J(F) - (\log k)  I(F) \geq (-c + o(1))  I(F).
% \end{equation}
% This can now be reformulated as 

\begin{equation}\label{final_intergral_num}
     \frac{ 2 \int_{-\infty}^{\infty} 
    e^{i \gamma \tau} (1+ i \tau)^{i \tau + 1}  N( \tau ) \, d \tau}{
    \int_{-\infty}^{\infty} e^{i \gamma \tau} (1 + i \tau)^{i \tau + 1} N_0 (\tau) \, d \tau}
    > -c + o(1),
\end{equation}
where we define $N_0 (\tau) := \int_0^{\infty} \widetilde{Q} (x)^2 \,
e^{- (1 + i \tau)x} \, dx.$ We numerically approximate the supremum of the quotient on the left-hand side of (\ref{final_intergral_num}) over polynomials of degree 20, and the best value we get is $-0.27357161702 \dots$. Once a good polynomial is selected, we freeze it and certify the result. An interested reader may find the details of these calculations in the notebooks attached to this work. 

As the table below shows, the gains appear to diminish as the polynomial degree increases. Also, the row for $\operatorname{deg} \widetilde{Q} = 0 $ serves as a nice sanity check to the first part of Theorem \ref{asymptotyczny_teoremat}.

\begin{table}[H]
    \centering
    {\sffamily\small
    \setlength{\tabcolsep}{18pt}
        \begin{tabular}{cc}
            \toprule
            Degree of $\widetilde{Q}$ & Best constants found \\
            \midrule
            0  & $-$0.336523 \\
            1  & $-$0.329521 \\
            2  & $-$0.284984 \\
            3  & $-$0.276507 \\
            4  & $-$0.275032 \\
            5  & $-$0.274489 \\
            10 & $-$0.273744 \\
            15 & $-$0.273615 \\
            20 & $-$0.273572 \\
            \bottomrule
        \end{tabular}
    }
    \caption{Numerical values of the best asymptotic constants obtained for polynomials of various degrees.}
    % \label{tab:eta_values}
\end{table}

% We define the following replacements for $I$ and $J$ to avoid carrying the normalizing constants all the time:

% It follows directly from (\ref{norm_I(F)}) and (\ref{norm_J(F)}) that we have

\section{Further remarks}

The results concerning bounds on $H_m$ and $M_k$ can be improved in at least a few ways. Some natural approaches are:

\begin{itemize}
    \item 
        We can consider a wider family of functions than (\ref{ansatz_1}) by introducing symmetric polynomials of higher degree. In principle, the form:
        \begin{equation*}
            F(t_1, \dots , t_k) = \frac{\Phi (y, P_2, P_3, \dots)}{\prod_{i=1}^k (1+ a t_i)},
        \end{equation*}
        where $\Phi$ is a polynomial, $y := 1 + \varepsilon - P_1$, and $P_r := \sum_{i=1}^k t_i^r$ should lead to stronger numerics.
    \item
        The error terms from the proof of Theorem \ref{asymptotyczny_teoremat} were bounded rather loosely. 
        It might be possible to extend this result to all $k$ by 
        providing estimates with explicit constants.
\end{itemize}

\clearpage

%%%%%%%%%%%%%%%%%%%%%%%%%%%%%%%%%%%%%%%%%%%%%%%%%%%%%%%%%%%%%%%%%%%%%%%
%%%%%%%%%%%%%%%%%%%%%%%%%%%%%%%%%%%%%%%%%%%%%%%%%%%%%%%%%%%%%%%%%%%%%%%

%\bibliography{references}

\end{document}